\documentclass[12pt, reqno]{amsart}
\usepackage[utf8]{inputenc} 
\usepackage[T1]{fontenc}
\usepackage[english]{babel}
\usepackage{mathtools}
\usepackage{amsmath}
\usepackage{amsfonts}
\usepackage[colorlinks, linkcolor=blue, filecolor=blue,
     citecolor = olive, urlcolor=purple]{hyperref}
\usepackage{amssymb}
\usepackage{amsthm}
\allowdisplaybreaks
\usepackage{mathrsfs}
\usepackage{scalerel}
\usepackage{orcidlink}

\usepackage{dsfont}
\usepackage{amsmath,amssymb,amsthm,amsfonts,enumerate,url,mathrsfs,tikz,graphicx,pifont,doi,comment}
\usepackage[normalem]{ulem}
\usepackage{epigraph}	
\usepackage{tikz,graphicx}
\usetikzlibrary{automata,positioning}
\usetikzlibrary{calc}

\usepackage{enumerate,url,float,lscape}
\usepackage{hyperref}
\usepackage{academicons}
\usepackage{xcolor}

\usepackage[normalem]{ulem}
\newcommand{\orcid}[1]{\href{https://orcid.org/#1}{\textcolor[HTML]{A6CE39}{\aiOrcid}}}

\allowdisplaybreaks

\usepackage{aliascnt}

\theoremstyle{plain}
\newtheorem{theorem}{Theorem}[section]

\newaliascnt{corollary}{theorem}
\newaliascnt{proposition}{theorem}
\newaliascnt{lemma}{theorem}

\newtheorem{lemma}[lemma]{Lemma}
\newtheorem{proposition}[proposition]{Proposition}
\newtheorem{corollary}[corollary]{Corollary}

\aliascntresetthe{corollary}
\aliascntresetthe{proposition}
\aliascntresetthe{lemma}

\theoremstyle{definition} 

\newaliascnt{definition}{theorem}
\newaliascnt{assum}{theorem}
\newaliascnt{assums}{theorem}
\newaliascnt{conjecture}{theorem}
\newaliascnt{conv}{theorem}

\newtheorem{definition}[definition]{Definition}

\newtheorem{conv}[conv]{Convention}

\aliascntresetthe{definition}
\aliascntresetthe{assum}
\aliascntresetthe{assums}
\aliascntresetthe{conjecture}
\aliascntresetthe{conv}

\newaliascnt{rems}{theorem}
\newaliascnt{rem}{theorem}
\newaliascnt{exa}{theorem}
\newaliascnt{exs}{theorem}

\newtheorem{rem}[rem]{Remark}

\aliascntresetthe{rems}
\aliascntresetthe{rem}
\aliascntresetthe{exa}
\aliascntresetthe{exs}

\numberwithin{equation}{section}
\numberwithin{lemma}{section}

 \def\mG{\mathsf{G}}

 \def\mV{\mathsf{V}}
 \def\mE{\mathsf{E}}

 \def\mP{\mathsf{P}}

 \def\mv{\mathsf{v}}
 \def\me{\mathsf{e}}
 \def\mf{\mathsf{f}}
 \def\mw{\mathsf{w}}

 \def\mf{\mathsf{f}}

\newcommand{\E}{\mathbb{E}}
\newcommand{\PP}{\mathbb{P}}

\newcommand{\N}{\mathbb{N}}
\newcommand{\heatcont}{{\mathcal Q}_t}
\newcommand{\mVD}{{\mathsf{V}_\mathrm{D}}}

\newcommand{\ud}{\,\mathrm{d}}
\newcommand{\e}{\mathrm{e}}

\newcommand{\Graph}{\mathcal{G}} 

\makeatletter
\@namedef{subjclassname@2020}{%
  \textup{2020} Mathematics Subject Classification}
\makeatother

\title[Faber--Krahn inequality for the heat content on quantum graphs]{Faber--Krahn inequality for the heat content on quantum graphs via random walk expansion}

\subjclass[2020]{Primary: 34B45. Secondary: 05C81, 49Q10}

\keywords{Heat content, Faber-Krahn inequality, Feynman-Kac formula, random walk expansion, shape optimisation}

\author[P.~Bifulco]{Patrizio Bifulco \orcidlink{https://orcid.org/0009-0004-0628-374X}} 

\address{Patrizio Bifulco, Lehrgebiet Analysis, Fakult\"at Mathematik und Informatik, Fern\-Universit\"at in Hagen, D-58084 Hagen, Germany}
\email{patrizio.bifulco@fernuni-hagen.de}

\author[M.~Täufer]{Matthias Täufer \orcidlink{https://orcid.org/0000-0001-8473-2310}}
\address{Matthias Täufer, Universit\'e Polytechnique Hauts-de-France, C\'ERAMATHS/DMATHS,
F-59313 -- Valenciennes Cedex 9 France}

\email{matthias.taufer@uphf.fr}

\thanks{
The first named author was supported by the Deutsche Forschungsgemeinschaft (Grant 397230547). 
Much of this work was done while the second named author was employed by FernUniversität in Hagen.
Both authors are indebted to Delio Mugnolo (Hagen) for fruitful and insightful discussions.
}

\begin{document}

\begin{abstract}
We study the \emph{heat content} on quantum graphs and investigate whether an analogon of the \emph{Rayleigh--Faber--Krahn inequality} holds. 
This means that heat content at time $T$ among graphs of equal volume would be maximized by intervals (the graph analogon of balls as in the classic Rayleigh--Faber--Krahn inequality).
We prove that this holds at \emph{extremal times}, that is at small and at large times.
For this, we employ two complementary approaches: In the large time regime, we rely on a spectral-theoretic approach, using \emph{Mercer's theorem} whereas the small-time regime is dealt with by a random walk approach using the \emph{Feynman-Kac formula} and Brownian motions on metric graphs.
In particular, in proving the latter, we develop a new expression for the heat content as a positive linear combination of expected return times of (discrete) random walks -- a formulation which seems to yield additional insights compared to previously available methods such as the celebrated \emph{Roth formula} and which is crucial for our proof.

The question whether a Rayleigh--Faber--Krahn inequality for the heat content on metric graphs holds at \emph{all times} remains open.

%
\end{abstract}
\maketitle

\section{Introduction}
Metric graphs are networks of intervals, joined at their end points to form a graph-like metric space.
Self-adjoint differential operators on metric graphs are called \emph{quantum graphs} and have become an important model in both spectral theory and analysis of differential or evolution equations, cf., \cite{BerKuc13, Mug14, Mug19, Kur23} and references therein. In particular, attention has been devoted to \emph{spectral geometry}, that is the investigation of the interplay between the topological and metric properties of $\Graph$ on the one hand, and Laplacian eigenvalues on the other hand, cf., \cite{Kur23}. 

Analogously to domains and manifolds, the heat equation on metric graphs
\begin{equation}
	\label{eq:heat}
\frac{\partial}{\partial t} u = \Delta u,
\end{equation}
where $\Delta$ is a suitably defined self-adjoint Laplace operator,
models diffusion of heat.
Fundamental solutions of~\eqref{eq:heat} are encoded in the so-called \emph{heat kernel} $(0,\infty) \times \Graph \times \Graph \ni (t,x,y) \mapsto  p_t^\Graph(x,y)$ depending on the space and time, which can be used to infer subtle information on the underlying metric graph. 
While the dependence of the spectrum on the geometry of the graph has been studied a lot, in particular by means of variational methods, the behavior of these fundamental solutions, or the heat kernel, is less understood. Indeed, since the heat kernel not depends on the eigenvalues but in particular on the shape of the eigenfunctions of the graph Laplacian $\Delta_\Graph$, it seems to be rather hard to deduce such information by purely variational techniques.

Therefore, instead of studying the dependence of the heat kernel itself on the geometry, one has considered a relaxation, namely its spatial integral on $\Graph \times \Graph$, known as the \emph{heat content} $\heatcont(\Graph)$ of $\Graph$. 
From a physical point of view, the heat content measures how much mass of a constant initial configuration of heat remains in the graph at time $t$ under the action of the heat semigroup $(\e^{-t\Delta_\Graph})_{t \geq 0}$ and it has been studied a lot already on domains and manifolds \cite{BerDav89, Gil99, Sav16}. 
On metric graphs, we are not aware of investigations of the dependence of the heat content on the geometry, but the time-integrated heat content, known as \emph{torsional rigidity} has been studied in \cite{MugPlu23,BifMug23b, OezcanT-24}.
The central observation for the torsional rigidity is that, similarly to the variational characterization of eigenvalues, the heat content also has a variational characterization making it amenable to spectral-geometric methods.

However, a similar variational representation for the heat content does not exist whence it requires different methods.
One different perspective on the heat kernel and thus the heat content is provided by \emph{Roth's formula}, see \cite{Rot84, KosPotSch07}, which represents the heat kernel at time $t$ as a sum over directed paths.
Indeed, Roth's formula is used to study the heat content in \cite{BifMug24}.
However, while Roth's formula is useful for deriving short-term asymptotics of the heat content, the sign-indefiniteness of terms appearing in it makes it challenging to achieve what is usually achieved via variational methods in spectral geometry: namely to \emph{compare} the heat content of different metric graphs.

Therefore, in this article, we resort to a different approach coming from stochastic analysis: using the theory of Brownian motions on metric graphs \cite{KosPotSch12}, it is possible to derive a \emph{Feynman-Kac-formula} for the heat kernel. This leads to a different understanding of $\heatcont(\Graph)$ through random walks on metric graphs and in particular discrete random walks on \emph{combinatorial graphs}. 
A common question in spectral geometry is the question of \emph{shape-optimization} leading to optimal geometries for the eigenvalues, also known as \emph{Rayleigh--Faber--Krahn inequality} (or briefly \emph{Faber--Krahn inequality}). 
It was proved, independently and with different methods by Nicaise, Friedlander and Kurasov--Naboko \cite{Nic87, Fri03, KurNab14} that the lowest non-trivial eigenvalue is \emph{minimized} by path graphs, see also~\cite{DufKenMug22} for analogous statements on infinite metric graphs.
The torsional rigidity behaves complementary to it eigenvalue and is \emph{maximized} by path graphs; this was verified in \cite[Theorem~4.6]{MugPlu23} for metric graphs and in \cite[Lemma~5.2]{BifMug24} for combinatorial graphs. 

In this article, we study the same question for the heat content.
Due to a lack of variational techniques, we can only prove asymptotic statements in the large and small-time regime in both of which we prove that the heat content is maximized by path graphs. 
Let us mention that on domains with Dirichlet boundary, it was shown in \cite{BurSch01} that the heat content on domains is \emph{maximized} on balls at all times.
However, since metric graphs behave fundamentally different from domains, for instance among the class of all graphs of a given volume, there is no polynomial bound on $\epsilon$-balls, it is not clear how to transfer these methods to metric graphs.

The structure of this article is as follows:
Section \ref{sec:notions_and_main_results} provides notation and definitions of the main objects used in this article.
It also contains~\autoref{thm:main-thm}, the Faber-Krahn inequality of the heat content at small and large times.
Section~\ref{sec:Feynman-Kac} introduces the Feynman-Kac formula for the heat content, and defines random walks on metric graphs. 
In particular,~\autoref{thm:densities-random-walk} therein collects useful properties on Brownian motions on metric graphs which are going to be used in the subsequent Section~\ref{sec:proof_small_times}.
Therein, a new path sum formula for the heat content in terms of expectations of return times of discrete random walks is provided in~\autoref{thm:probabilistic-heat-content-formula}.
This is then used to prove the Faber-Krahn inequality for the heat content at small times -- the more difficult part of~\autoref{thm:main-thm}.
The Faber-Krahn inequality for the heat content at large times is proved using a different, more spectral theoretic approach in Section~\ref{sec:proof_large_times}.
\section{Notions and main results}
\label{sec:notions_and_main_results}
\subsection{Basic notions on metric graphs}
Let $\Graph$ be a compact and connected metric graph over a combinatorial graph $\mG = (\mV, \mE)$ consisting of a \emph{finite} vertex set $\mV = \mV(\mG)$ and a \emph{finite} edge set $\mE = \mE(\mG)$, that is the edges $\me \in \mE$ are understood as  intervals and topologically glued at their end points according to the structure of $\mG$.
We refer to \cite{BerKuc13, Mug14, Mug19, Kur23} and references therein for a more comprehensive introduction.
The length of an edge $\me \in \mE$ is denoted by $\ell_\me \in (0, \infty)$, and the \emph{total length} $\vert \Graph \vert$ of $\Graph$ by
\[
\vert \Graph \vert := 
\sum_{\me \in \mE}
\ell_\me \in (0,\infty).
\]
An edge $\me$ is \emph{incident} to a vertex $\mv$, written $\me \sim \mv$, if $\me$ is connect $\mv$ with another vertex in the combinatorial graph $\mG$, or, in other words, if $\mv$ corresponds to $0$ or $\ell_\me$ in the metric graph $\Graph$. Two distinct vertices $\mv,\mw \in \mV$ are called \emph{adjacent} if there exists an edge $\me \in \mE$ incident to both $\mv$ and $\mw$. Moreover, for a vertex $\mv \in \mE$ we denote by $\deg(\mv) := \# \{ \me \in \mE \colon \me \sim \mv \}$ its \emph{degree}. 
A \emph{path} is a sequence of vertices $(\mv_0,\mv_1,\dots,\mv_n)$ such that $\mv_i \sim \mv_{i+1}$ for every $i=0,1,\dots,n-1$.
We say that $\Graph$ is \emph{connected} if for every pair $\mv, \mw \in \mV$ there is a path joining $\mv$ and $\mw$.

For our purposes, it will be sometimes crucial to specialize to the class of metric graphs having \emph{rationally dependent} edge lengths, that is,
\[
\frac{\ell_\me}{\ell_\mf} \in \mathbb{Q} \qquad \text{for any two edges $\me,\mf \in \mE$.}
\]
Moreover, we call a metric graph $\Graph$ \emph{equilateral} if $\ell_\me = \ell$ for all $\me \in \mE$ for some $\ell > 0$. 
Clearly, graphs with rationally dependent edge lengths can be made equilateral by subdividing all edges with the introduction of additional \emph{dummy vertices}.

\subsection{Dirichlet Laplacian and the heat content}
From now on, let $\mV_{\mathrm{D}} \subset \mV$ be a non-empty subset which will be referred to as the set of \emph{Dirichlet vertices}. 
Moreover, on metric graphs, when we speak about connectedness in the presence of a set of Dirichlet vertices, we always mean that $\Graph \setminus \mVD$ is connected.
Indeed, otherwise, the Laplacian with Dirichlet conditions at $\mVD$ will be a direct sum of Laplacians on the connected components.

\begin{definition}
We denote by $\Delta_{\Graph}^{\mV_{\mathrm{D}}} = \Delta_\Graph$ the Laplacian on $L^2(\Graph)$ with Dirichlet conditions at all vertices in $\mV_{\mathrm{D}}$ and Kirchhoff and continuity conditions at all vertices in $\mV \setminus \mV_{\mathrm{D}}$, i.e., the self-adjoint operator on 
\[
L^2(\Graph) := \bigoplus_{\me \in \mE} L^2(0,\ell_\me)
\]
 associated with the quadratic form \(a=a_{\Graph}\) given by
	\[a_{\Graph}(u):=\int_{\Graph}|u'(x)|^2\ud x=\sum_{\me\in \mE}\int_0^{\ell_\me}|u_\me'(x_\me)|^2\ud x_\me\]
on the form domain
	\[H^1_0(\Graph;\mV_\mathrm{D}):=\left\{u=(u_\me)_{\me\in \mE}\in\bigoplus_{\me\in \mE} H^1(0,\ell_\me) \: : \: \begin{array}{l} u(\mv)=0\text{ for }\mv\in \mV_\mathrm{D},\\ u\text{ is continuous in every }\mv\in \mV \setminus \mV_\mathrm{D} \end{array}\right\}.\]
We also denote the \emph{path graph of total length $\mathcal{L} > 0$} by $\mathcal{P}_\mathcal{L}$ and the Laplacian on it with a Dirichlet condition on one side and a Neumann condition on the other side by $\Delta_{\mathcal{P}_\mathcal{L}}^{\{0 \}}$. 
\end{definition}

It is well-known that the Dirichlet Laplacian $\Delta_\Graph^\mVD$ generates an analytic $C_0$-semigroup -- known as \emph{heat semigroup} -- which will be denoted with $(\mathrm{e}^{t\Delta_\Graph})_{t \geq 0}$ in the sequel. 
This leads us to the main quantity of this paper:

\begin{definition}\label{def:heat-content}
For a compact metric graph $\Graph$ with $\mV_{\mathrm{D}} \subset \mV$ such that $\Graph \setminus \mV_{\mathrm{D}}$ is connected,
the \emph{heat content} of $\Graph$ at time $t>0$ is 
\[
	\heatcont(\Graph; \mV_{\mathrm{D}})
	:=
	\big\lVert \e^{t \Delta_\Graph^\mVD} \mathbf{1} \big\rVert_{L^1(\Graph)}.
\]
\end{definition}

\begin{rem}
	Note that $\e^{t \Delta_\Graph^\mVD} \mathbf{1}$ is simply the value at time $t$ of the heat equation with initial condition constant to one, that is the solution of
	\[
		\begin{cases}
		\frac{\partial}{\partial t} u &= \Delta_\Graph^\mVD u,\\
		u(0) &\equiv \mathbf{1},
		\end{cases}
	\]
	which means that $\heatcont(\Graph; \mV_{\mathrm{D}})$ measures the $L^1$-norm of heat remaining in $\Graph$ at time $t$ when an initially constant distribution of heat dissipates through the set of Dirichlet vertices.
\end{rem}
The heat equation is also closely related to the concept of \emph{torsional rigidity} $T(\Graph; \mVD)$, that is the $L^1$-norm of the solution of
	\[
		-\Delta_\Graph^\mVD u = \mathbf{1}.
	\]
	Note that in the case of torsional rigidity, is is known that a Faber--Krahn inequality holds on metric graphs.
	More precisely, among graphs of a given volume, $T(\Graph; \mVD)$ is maximized by path graphs with a Dirichlet condition on one side~\cite{MugPlu23} and a Kirchhoff-Neumann condition on the other side, see also~\cite{OezcanT-24} for an analogous result in the case of $\delta$-vertex conditions.
	
	Both the lowest eigenvalue (cf.\ \eqref{eq:eigenvalues}) as well as the torsional rigidity enjoy \emph{variational formulations} as minimizer of the \emph{Rayleigh quotient}
	\[
	\lambda_1(\Graph; \mVD)
	=
	\min_{0 \neq u \in H^1_0(\Graph; \mVD)}
	\frac
	{\lVert u' \rVert_{L^2(\Graph)}^2}
	{\lVert u \rVert_{L^2(\Graph)}^2}
	\]
	or maximizer of the \emph{P{\'o}lya quotient}
	\[
	T(\Graph; \mV_{\mathrm{D}})
	=
	\max_{0 \neq u \in H^1_0(\Graph; \mVD)}
	\frac
	{\lVert u \rVert_{L^1(\Graph)}^2}
	{\lVert u' \rVert_{L^2(\Graph)}^2},	
	\]
	and these are central tools for comparing these quantities on different metric graphs.
	For the heat content, a comparable variational formulation is not known.
	However, one has at least
	\[ 
	T(\Graph; \mV_{\mathrm{D}})
	=
	\int_0^\infty
	\heatcont (\Graph; \mV_{\mathrm{D}})
	\]
	raising the question whether a similar Faber--Krahn inequality already holds for $\heatcont$.

\begin{rem}
	\label{rem:one_Dirichlet_vertex} 
  As the heat semigroup $(\mathrm{e}^{t\Delta_\Graph^{\mV_\mathrm{D}}})_{t \geq 0}$ maps $L^2(\Graph)$ into the form domain $H^1_0(\Graph;\mV_\mathrm{D})$ due to analyticity, each $\mathrm{e}^{t\Delta_\Graph^{\mV_\mathrm{D}}}$ is a bounded operator from $L^2(\Graph)$ to $L^\infty(\Graph)$, and by duality, from $L^1(\Graph)$ to $L^\infty(\Graph)$. 
 Thus, each $\mathrm{e}^{t\Delta_\Graph^{\mV_\mathrm{D}}}$, $t > 0$, is an integral operator with correponding integral kernel $p_t^\Graph \in L^\infty(\Graph \times \Graph)$, cf.\ also \cite{KosMugNic22, BifMug23} for stronger regularity properties on $p_t^\Graph$, known as the \emph{heat kernel}, i.e., 
	 \[
	 \big(\e^{t\Delta_\Graph}f \big)(x) = \int_\Graph p_t^\Graph(x,y) f(y) \ud y \qquad \text{for all $f \in L^2(\Graph)$ and all $x \in \Graph$.}
	 \]
	 According to \autoref{def:heat-content}, the heat content $\heatcont(\Graph;\mV_\mathrm{D})$ at time $t>0$ can alternatively be represented as the $L^1$ norm of $p_t^\Graph$, cf.~\cite{BifMug24} for more details.\\
\end{rem}

\subsection{The extremal Faber--Krahn property and main results}
We can now introduce the central notions used in this paper.
\begin{definition}[Extremal Faber--Krahn property]
Let $\Graph$ be a compact metric graph, let $\emptyset \neq \mV_{\mathrm{D}} \subset \mV$ such that $\Graph \setminus \mV_{\mathrm{D}}$ is connected, and let $t>0$. We say that $t \mapsto \heatcont(\Graph;\mVD)$ satisfies
\begin{enumerate}[(i)]
 \item\label{item:fk-small} the \emph{Faber--Krahn inequality for small times} if there exists a $t_0 > 0$ such that
 \[
 \heatcont(\Graph; \mV_{\mathrm{D}}) \leq \heatcont(\mathcal{P}_{\rvert \Graph \rvert}; \{ 0 \}) \qquad \text{for all $0 < t \leq t_0$},
\]
and equality holds for all $0 < t \leq t_0$ if and only if $\Graph = \mathcal{P}_{\rvert \Graph \rvert}$;
 \item\label{item:fk-large}
 the \emph{Faber--Krahn inequality for large times} if there exists a $T_0 > 0$ such that
 \[
 \heatcont(\Graph; \mV_{\mathrm{D}}) \leq \heatcont(\mathcal{P}_{\rvert \Graph \rvert}; \{ 0 \}) \qquad \text{for all $t \geq T_0$},
 \]
 and equality holds for all $t \geq T_0$ if and only if $\Graph = \mathcal{P}_{\rvert \Graph \rvert}$.
\end{enumerate}
 If $t \mapsto \heatcont(\Graph;\mVD)$ satisfies both \eqref{item:fk-small} and \eqref{item:fk-large}, we say that $t \mapsto \heatcont(\Graph;\mVD)$ (or simply the heat content) satiesfies an \emph{extremal Faber--Krahn inequality}.
\end{definition}

Our first main result is:
\begin{theorem}\label{thm:main-thm}
Let $\Graph$ be a compact metric graph, and let $\emptyset \subsetneq \mV_{\mathrm{D}} \subset \mV$ such that
$\Graph \setminus \mV_{\mathrm{D}}$ is connected
Then:
\begin{enumerate}[(i)]
\item\label{item:fk-large-thm} $t \mapsto \heatcont(\Graph;\mVD)$ satisfies the Faber--Krahn inequality for large times.
\item\label{item:fk-small-thm} If additionally the edge lengths of $\Graph$ are rationally dependent, then $t \mapsto \heatcont(\Graph;\mVD)$ satisfies the Faber--Krahn inequality for small times. 
In particular, the heat content satisfies an extremal Faber--Krahn inequality.
\end{enumerate}
Equality holds at sufficiently small or large times if and only if $\Graph = \mathcal{P}_{\vert \Graph \vert}$ with $\mv_{\mathrm{D}}$ being either the vertex identified with $0$ or with $\lvert \Graph \rvert$.
\end{theorem}
\begin{rem}
We emphasize that the statement of \autoref{thm:main-thm} is somewhat different to the classical Faber--Krahn inequality for the heat content on domains which holds for \emph{all} $t>0$ and was established in \cite{BurSch01}: indeed, there the heat content induced by the Laplace operator on a domain with fixed volume and with pure Dirichlet boundary conditions is compared to the one induced by the ball with Dirichlet boundary conditions and of same volume. Translating this to the language of graphs, this corresponds to the comparison of a graph, where
\[
\emptyset \neq \mVD = \{ \mv \in \mV \: : \: \deg(\mv) = 1 \}
\]
with an interval of same total length and with \emph{two} Dirichlet vertices in terms of their heat content. Therefore, it is not clear whether the methods used in \cite{BurSch01} can be adapted to metric graphs. 
\end{rem}

\begin{rem}
	Note that~\autoref{thm:main-thm} does \emph{not} state the existence of some $t_0$ above or below which the heat content of any graph of a given total length will be dominated by the heat content of the corresponding path graph.
	Indeed, it does not even exclude the existence of a sequence of graphs $(\Graph_n)_{n \in \N}$ with times $t_n \searrow 0$ such that
	\[
	\mathcal{Q}_{t_n}(\Graph_n; \mV_{\mathrm{D}}) 
	> \mathcal{Q}_{t_n}(\mathcal{P}_{\rvert \Graph \rvert}; \{ 0 \})
	\quad
	\text{for all $n \in \N$}.
	\]
	While we hold this to be unlikely, see also~\autoref{rem:conjectuere} below, a closer inspection of the proof of \autoref{thm:main-thm}~\eqref{item:fk-small-thm} suggests that the most problematic candidates in the small time regime might be pitchfork shaped graphs, consisting of a long segment with the Dirichlet point on one side and two short stubs attached on the other side.
\end{rem}
\begin{rem}[A very cautious conjecture]
\label{rem:conjectuere}
We believe that the Faber--Krahn property on metric graphs should hold at all times but are still reluctant to call this a conjecture.
Indeed, both the Faber--Krahn inequality for small times and large times use very different methods both of which rely on asymptotics of certain expressions.

In particular, the proof of the Faber--Krahn inequality for small times uses a somewhat surprising and delicate argument using that certain infinite sums can be identified with expectations of discrete random walks which are themselves independent of the topology of the underlying graph.
However, this argument only works asymptotically for sufficiently small times.

Also numerical simulations to gain insights as to whether the Faber--Krahn inequality holds at all times are a touchy issue for since, near $t = 0$, the heat content is in leading orders given by the volume of $\Graph$ and the number of Dirichlet vertices whereas the topology only enters as an exponentially small correction, cf.~\cite{BifMug24}.

In any case further research will be required to settle whether the Faber--Krahn-inequality holds for all times.
\end{rem}

Identifying all vertices in $\mV_{\mathrm{D}}$ with one \emph{single} vertex $\mv_{\mathrm{D}}$ will not affect the spectrum or eigenfunctions -- nor the heat content, defined above -- of $\Delta_{\Graph}$. Thus, from now on we stipulate:

\begin{conv}\label{conv:dirichlet-singleton}
The set of Dirichlet vertices $\mVD$ is a singleton, i.e., $\mV_{\mathrm{D}} = \{ \mv_{\mathrm{D}} \}$ for one single vertex $\mv_{\mathrm{D}} \in \mV$ and we abbreviate $\heatcont(\Graph;\mV_\mathrm{D}) := \heatcont(\Graph;\mv_\mathrm{D})$ for $t>0$. 
\end{conv}

\section{Feynman--Kac formula and Brownian motion on metric graphs}
\label{sec:Feynman-Kac}
We now lay the groundwork for the proof of \autoref{thm:main-thm}~\eqref{item:fk-small-thm}, the Faber--Krahn inequality at small times. The proof uses a novel expression for the heat content given in \autoref{thm:probabilistic-heat-content-formula}. It arises from the famous Feynman--Kac formula. For this purpose, we need notation on random walks, more precisely, Brownian motions on metric graphs, cf.\ also \cite{KosPotSch12}.

\begin{definition}[Brownian motions on metric graphs]
	The \emph{Brownian motion on $\Graph$, starting at $x \in \Graph$} is the stochastic process on $\Graph$ with continuous paths $(B_t)_{t \geq 0}$ such that
	\begin{itemize}
		\item[(i)]
		$B_0 = x$.
		\item[(ii)]
		While remaining within an edge, $B_t$ is an (ordinary one-dimensional) Brownian motion on an interval.
		\item[(iii)]
		Whenever $B_t \in \mV$ (i.e. the Brownian motion hits a vertex), $B_t$ independently chooses one of the neighbouring edges and continues as a one-dimensional Brownian motion on it until it hits $\mV$ for the next time. 
	\end{itemize}
	
	We denote the probability with respect to the Brownian motion, starting at $x \in \Graph$ by $\PP_x$.
	Furthermore, for a vertex $\mv \in \mV$, the first hitting time of $\mv$ is
	\[
	\tau_{\mv}^{\mathrm{BM}}
	:=
	\inf \{ t \geq 0 \colon B_t = \mv \}.
	\] 
\end{definition}

We shall use:

\begin{proposition}[Feynman-Kac formula]
	Let $\Graph$ be a compact graph and let $\{ \mv_\mathrm{D} \} = \mV_\mathrm{D} \subset \mV$ such that $\Graph \setminus \mV_{\mathrm{D}}$ is connected
	Then:
	\begin{equation}
		\label{eq:Feynman-Kac}
		\heatcont(\Graph; \mv_{\mathrm{D}})
		=
		\int_\Graph
		\PP_x \left[ \tau_{\mv_{\mathrm{D}}}^{\mathrm{BM}} \geq t \right]
		\mathrm{d} x
		\quad
		\text{for all $t > 0$}.
	\end{equation}
\end{proposition} 
\begin{proof}
According to \cite[Proposition~8.2]{BecGreMug21}
\begin{align*}
\big(\e^{t\Delta_\Graph} \mathbf{1} \big)(x) = \mathbb{E}_x\Big[\mathbf{1}_{\{\tau_{\mv_\mathrm{D}}^{\mathrm{BM}} \geq t \}}(B_t) \Big] \quad \text{for all $x \in \Graph$ and all $t>0$,}
\end{align*}
whence
\begin{align*}
\heatcont(\Graph;\mv_\mathrm{D}) &= \int_\Graph \mathbb{E}_x\Big[\mathbf{1}_{\{\tau_{\mv_\mathrm{D}}^{\mathrm{BM}} \geq t\}}(B_t) \Big] \ud x = \int_\Graph \PP_x\left[\tau_{\mv_\mathrm{D}}^{\mathrm{BM}} \geq t \right] \ud x \quad \text{for all $t>0$.} \qedhere
\end{align*}
\end{proof}

Let us introduce some more notation in order to better understand the Brownian motion on $\Graph$. 
To this end, set $\mathbb{N} := \{1,2,\dots\}$, and $\mathbb{N}_0 := \mathbb{N} \cup \{0\}$.

\begin{definition}[Stopping times and process of new vertices to be hit]
	Let $\Graph$ be a compact and connected metric graph and $x \in \Graph$.
	For the Brownian motion $(B_t)_{t \geq 0}$, starting at $x$, let
	\[
	\eta_0
	:=
	\inf 
	\{ t \geq 0 \colon B_t \in \mV \} \geq 0,
	\quad
	\text{and}
	\quad
	\eta_{n+1} 
	:=
	\inf \{ t > \eta_n \colon B_t \in \mV \setminus \{ \mv_n \} \}
	\quad
	\text{for $n \in \mathbb{N}_0$},
	\]
	where, we define
	\[
	\mv_n := B_{\eta_n}, 
	\quad  \text{for $n \in \mathbb{N}_0$}.
	\]
\end{definition}

Having fixed $x$ as the starting point, $\eta_0$ denotes the first hitting time of $(B_t)_{t \geq 0}$ in the vertex set $\mV$, and $\mv_0$ is the first vertex to be hit.
After $\eta_n$, the path will enter any of the adjacent edges of $\mv_n$ with equal probability and turn into a one-dimensional Brownian motion on this edge until it either hits $\mv_n$ again or reaches the other end of the edge which is incident to $\mv_n$.
Indeed, due to known properties of the Brownian motion, the path will hit $\mv_0$ infinitely often in a suitable neighbourhood of $\eta_n$, almost surely, cf., for instance, \cite{BonDov23}.
If it hits $\mv_n$ again, it will again continue as a Brownian motion on one of the adjacent edges with equal probability.
The choice of the next edge, and the motion between two intersections with the vertex set are mutually independent.
 
For the rest of this section, we focus on the case where $\Graph$ is \emph{equilateral}, that is, all edges of $\Graph$ have the same length $\ell > 0$. 
Ignoring which edge it chooses and tracking only the distance to $\mv_0$, the resulting process will be the absolute value of a one-dimensional Brownian motion, started at $0$ until it reaches the other side of one edge, that is, until the absolute value of a Brownian motion exceeds $\ell$, and hits $\mv_{n+1}$.
By symmetry (recall that we assume all edges to be equilateral) and independence of the choice of the next edge at every return to $\mv$, we have 
\[
	\PP_x
	\left[
	B_{\eta_{n+1}} = \mw 
	\mid
	B_{\eta_{n}} = \mv
	\right]
	=
	\begin{cases}
		\frac{1}{\deg(\mv)},
		&\ \text{if $\mv \sim \mw$},\\
		0, 
		&\ \text{else}. 
	\end{cases}
\]
In other words, the process $\{\mv_0, \mv_1, \mv_2, \dots \}$ of new vertices to be hit is a \emph{symmetric, discrete random walk} on the combinatorial graph $\mG$.
Let us formally introduce it:
\begin{definition}[Symmetric random walk]\label{def:symmetric-random-walk}
	Let $\mG = (\mV, \mE)$ be a combinatorial graph.
	The \emph{symmetric, discrete random walk}, starting at $\mv_0 \in \mG$, is the process $(\mv_0, \mv_1,\mv_2, \dots) = (\mv_n)_{n \in \mathbb{N}}$ such that for $n \in \mathbb{N}_0$,
	\[
	\PP
	\left[
	\mv_{n+1} = \mw
	\right]
	=
	\begin{cases}
	\frac{1}{\deg( \mv_n )},
	&
	\quad
	\text{if $\mv_n \sim \mw$},\\
	0,
	&
	\quad
	\text{else}.
	\end{cases}
	\]
	Given a starting vertex $\mv_0 \in \mV$ and a target vertex $\mw \in \mV$, we denote by
	\[
	\tau_\mw^\mG
	:= \tau_\mw:=
	\inf 
	\{
	n \in \mathbb{N}
	\colon
	\mv_n = \mw \}
	\]
	the \emph{first hitting time} of the process $(\mv_n)_{n \in \mathbb{N}_0}$ in $\mw$.
	If we want to emphasize that the process starts at $\mv_0$, we will write
	$\PP_{\mv_0}$ and $\E_{\mv_0}$ for the probability and the expectation with respect to this process.
\end{definition}

\begin{rem}[A warning about notation]
	Recall that the symbol $\tau_{\mv}^{\mathrm{BM}} \in (0, \infty)$ is the first hitting time with respect to the Brownian motion (on the metric graph $\Graph$) whereas $\tau_{\mv} \in \mathbb{N}$ denotes the first hitting time with respect to the symmetric random walk on $\mG$.
	
	Furthermore, we use the notation $\PP_x$ and $\E_x$ for probability and expectation with respect to the Brownian motion started at $x$, and the notation $\PP_{\mv}$ and $\E_{\mv}$ for probability and expectation with respect to the symmetric random walk, respectively.
	However, it will be clear from the context which one is meant. 
\end{rem}

Using these observations, we can state and prove:

\begin{theorem}\label{thm:densities-random-walk}
	Let $\Graph$ be a compact and connected metric graph and assume that $\Graph$ is \emph{equilateral}, i.e., there is $\ell > 0$ such that $\ell_e = \ell$ for all $\me \in \mE$.
	Let $x \in \Graph$ correspond to a position $x \in [0,\ell]$ on an edge $\me \simeq [0,\ell]$ connecting vertices $\mv, \mw \in \mV$, identified with $0$ and $\ell$, respectively. 
	\begin{enumerate}[(i)]
	\item
	For the Brownian motion starting at $x \in \Graph$, the process of hitting times of new vertices
	\[
	( \eta_0, \eta_1, \eta_2, \dots ) =: (\eta_n)_{n \in \mathbb{N}}
	\]
	has independent increments $\eta_{n+1} - \eta_n$ where only the distribution of $\tau_0$ depends on $x$, namely 
	\begin{align}\label{eq:first-hitting-time-density}
	\PP_x
	\left[
		\eta_0 \geq t
	\right]
	=
	\sum_{n=0}^\infty \frac{4}{(2n+1)\pi} \sin\bigg( \frac{(2n+1)\pi x}{\ell} \bigg)\e^{-t\frac{(2n+1)^2\pi^2}{2\ell^2}},
	\end{align}
	for all $x \in [0,\ell]$ and $t>0$, whereas
	\begin{align}\label{eq:second-plus-hitting-time-density}
	\PP
	\left[ 
		\eta_{n+1} - \eta_n \geq t 
	\right]
	=
	\frac{2}{\sqrt{\pi}} \int_{-\infty}^{\frac{\ell}{\sqrt{2t}}} \e^{-x^2} \ud x,
	\end{align}
	in particular
	\begin{align}\label{eq:limit-eta-n-in-difference}
	\PP[\eta_{n+1} - \eta_n < t] = \mathrm{erfc}\bigg(\frac{\ell}{\sqrt{2t}} \bigg) \quad \text{for any $n \in \mathbb{N}_0$ and all $t>0$,}
	\end{align}
	where $\mathrm{erfc}(a) := \frac{2}{\sqrt{\pi}} \int_a^\infty \e^{-x^2} \ud x$ denotes the complementary error function.
	\item
	The first vertex $\mv_0$ to be hit satisfies
	\[
	\PP_x \left[ \mv_0 = \mv \right]
	=
	1 - 	\PP_x \left[ \mv_0 = \mw \right],
	\]
	and in particular, by symmetry,
	\[
	\int_0^\ell
	\PP_x \left[ \mv_0 = \mv \right]
	\mathrm{d}x
	=
	\int_0^\ell
	\PP_x \left[ \mv_0 = \mw \right]
	\mathrm{d}x
	=
	\frac{\ell}{2}, 
	\]
	as well as
	\[
	\int_\Graph
	\PP_x \left[ \mv_0 = \mv \right]
	\mathrm{d} x
	=
	\frac{\ell \deg(\mv)}{2}
	\quad
	\text{for every $\mv \in \mV$}.
	\]
	\item
	The process $(\mv_0, \mv_1, \dots) := (\mv_n)_{n \in \mathbb{N}}$ is a symmetric random walk on $\mG$, starting at $\mv_0$.
	\end{enumerate}
\end{theorem}

\begin{proof}
	(i) It is known (cf.~\cite{KarShr98}) that $u(t,x) := \PP_x[\eta_ 0 \geq t]$, $x \in [0,\ell]$, $t>0$ solves the initial value problem
	\begin{equation}\label{eq:initial-value-problem}
\begin{cases} \frac{\partial}{\partial t} u(t,x) &= \frac{1}{2} \frac{\partial^2}{\partial x^2} u(t,x), \\
	u(t,0) &= u(t,\ell) = 0, \\
	\lim_{t \rightarrow 0^+} u(t,x) &= 1, 
	\end{cases} 
\qquad x \in [0,\ell], \:\: t> 0.
	\end{equation}
	Analogously to \cite[Proposition~6.3]{BorHarJon22} or by computing the Fourier series, one concludes
	\[
	\PP_x[\eta_ 0 \geq t] = \sum_{k=1}^\infty c_k \sin\bigg(\frac{k\pi x}{\ell} \bigg)\e^{-t\frac{k^2\pi^2}{2\ell^2}}, \quad x \in \Graph, \: t > 0,
	\]
	with
	\begin{align}\label{eq:fourier-coefficients}
	c_k \e^{-t\frac{k^2\pi^2}{2\ell^2}} = \frac{2}{\ell} \int_0^\ell u(t,x) \sin\bigg( \frac{k\pi x}{\ell} \bigg) \ud x, \quad k \in \mathbb{N}.
	\end{align}
	Taking $t \rightarrow 0^+$ on both sides of \eqref{eq:fourier-coefficients}, it follows that for any $k \in \mathbb{N}$
	\begin{align*}
	c_k = \frac{2}{\ell} \int_0^\ell \sin\bigg(\frac{k\pi x}{\ell} \bigg) \ud x = \frac{2}{k \pi} \left[1-(-1)^k \right] = \begin{cases} \frac{4}{k\pi}, & \text{if $k$ is odd,} \\ 0, & \text{if $k$ is even,} \end{cases}
	\end{align*}
	yielding~\eqref{eq:first-hitting-time-density} for $\mathbb{P}_x[\eta_0 \geq t]$ for $x \in \Graph$ and $t>0$.
	
	The second identity follows from the \emph{reflection principle}, cf., for instance, \cite{KarShr98}. 
	Indeed, as every edge $\me \in \mE$ has length $\ell > 0$, and $(B_t)_{t \geq 0}$ behaves like a one-dimensional Brownian motion starting at $0$ on the interval $[0,\ell]$ after reaching $\eta_n$ and deciding for any of the adjacent edges to $\mv_n$, the reflection principle yields -- as $B_t \sim \mathcal{N}(0,t)$ for any $t>0$ -- that
	\begin{align*}
	\PP[\eta_{n+1} - \eta_n \geq t] &= 1 - \PP[\eta_{n+1} - \eta_n < t] = 1- 2\PP[B_t \geq \ell] \\&= 1- \sqrt{\frac{2}{\pi t}} \int_\ell^\infty \e^{-\frac{x^2}{2t}} \ud x = 1- \frac{2}{\sqrt{\pi}}\int_{\frac{\ell}{\sqrt{2t}}}^\infty \e^{-x^2}\ud x = \frac{2}{\sqrt{\pi}} \int_{-\infty}^{\frac{\ell}{\sqrt{2t}}} \e^{-x^2} \ud x,
	\end{align*}
	which is the identity claimed in \eqref{eq:second-plus-hitting-time-density}.

	(ii) The first identity is clear. Moreover, by symmetry this implies
	\begin{align*}
	\int_0^\ell \PP_x[\mv_0 = \mv] \ud x = \ell - \int_0^\ell \PP_x[\mv_0 =\mw] \ud x = \ell-\int_0^\ell \PP_x[\mv_0 = \mv] \ud x,
	\end{align*}
	yielding the second identity. 
	The third identity  follows from the fact that $\Graph$ is equilateral.
	
	(iii) This is due to the construction of the random walk $(\mv_n)_{n \in \mathbb{N}} = (\mv_1,\mv_2,\dots)$, using again that $\Graph$ is equilateral.
\end{proof}

We next recall a classic result on symmetric, discrete random walks, see for instance \cite{Lov93}.

\begin{proposition}\label{prop:expectation_random_walk}
	Let $\mG$ be a finite and connected combinatorial graph.
	Then, for any starting vertex $\mv_0 \in \mV$, the random variable $\tau_{\mv_0}$ is an integrable random variable with
	\[
	\E_{\mv_0}
	\left[
		\tau_{\mv_0}
	\right]
	=
	\sum_{n = 1}^\infty
	n	
	\PP_{\mv_0}
	\left[
		\tau_{\mv_0} = n
	\right]
	=
	\frac{2	\# \mE}{\deg( \mv_0 )},
	\]
	where $\# \mE$ denotes the cardinality of $\mE$.
\end{proposition}

\begin{rem}
	In our proof of the Faber--Krahn inequality for small times, Proposition~\ref{prop:expectation_random_walk} seems to be the key ``new'' ingredient, at least from the perspective of path expansions of the heat content.
	It will allow to explicitly calculate certain infinite sums arising in the path expansions arising in \autoref{thm:probabilistic-heat-content-formula} instead of estimating them.
	We should also emphasize that this is a particularly ``lucky'' circumstance, related to the fact that an explicit and very simple expression for the expectation of the return time of the discrete, symmetric random walks exist.
\end{rem}

\section{Proof of the Faber--Krahn inequality for small times}
\label{sec:proof_small_times}

We now formulate a probabilistic expression for the heat content on equilateral metric graphs.

\begin{theorem}\label{thm:probabilistic-heat-content-formula}
	Let $\Graph$ be a compact, equilateral metric graph and let $\emptyset \subsetneq \mV_\mathrm{D}\subset \mV$.
	Then, for any $t > 0$, the heat content of $\Graph$ at time $t > 0$ satisfies
	\begin{equation}
		\label{eq:probabilistic_expression_heat_content}
		\heatcont(\Graph; \mv_{\mathrm{D}} )
		=
		\frac{\deg(\mv_{\mathrm{D}})}{2}
		\left[
		\alpha_0(t)
		\E_{\mv_{\mathrm{D}}}[\tau_{\mv_{\mathrm{D}}}]
		+
		\sum_{n = 1}^\infty
		\left(
		\alpha_{n}(t) - \alpha_{n-1}(t)
		\right)
		\E_{\mv_{\mathrm{D}}}\left[(\tau_{\mv_{\mathrm{D}}}-n) \mathbf{1}_{\{\tau_{\mv_{\mathrm{D}}} \geq n+1\}} \right]		
		\right],
	\end{equation}
	where
	\begin{equation}
	\label{eq:defintition_alpha}
	\alpha_n(t)
	:=
	\int_0^\ell
	\mathbb{P}_x
	\left[
		\eta_0
		+
		\sum_{k = 1}^n
		(\eta_k - \eta_{k-1})
		\geq t
	\right]
	\mathrm{d}x.
	\end{equation}
\end{theorem}
\begin{rem}
\begin{enumerate}[(i)]
\item By the identification of \emph{all} Dirichlet vertices as suggested in \autoref{conv:dirichlet-singleton}, we observe for a general set of Dirichlet vertices $\mV_\mathrm{D} \subset \mV$, all of which are of degree one, that the heat content at time $t>0$ can be represented as
\begin{align*}
\heatcont(\Graph; \mV_\mathrm{D})
		=
		\frac{\# \mV_\mathrm{D}}{2}
		\left[
		\alpha_0(t)
		\E_{\mV_{\mathrm{D}}}[\tau_{\mV_{\mathrm{D}}}]
		+
		\sum_{n = 1}^\infty
		\left(
		\alpha_{n}(t) - \alpha_{n-1}(t)
		\right)
		\E_{\mV_{\mathrm{D}}}\Big[(\tau_{\mV_{\mathrm{D}}} -n)\mathbf{1}_{\{\tau_{\mV_{\mathrm{D}}} \geq n+1\}} \Big] 		
		\right],
\end{align*}
where $\tau_{\mV_\mathrm{D}}$ represents the first hitting time of the Dirichlet vertex set with respect to the symmetric random walk and $\# \mV_\mathrm{D}$ the cardinality of $\mV_\mathrm{D}$. 
\item Let us also note that other expressions for the heat content on metric graphs \ can be derived using the \emph{Roth formula} for the heat kernel~\cite{Rot83, Rot84,BecGreMug21, BorHarJon22, KosPotSch07, Cat98, Nic97}: this is done in \cite{BifMug24}.
	However, our expression, spelled out in Theorem~\ref{thm:probabilistic-heat-content-formula} below, has the advantage that it will allow to invoke~\autoref{prop:expectation_random_walk}.
	The heat content formula obtained in \cite{BifMug24} will carry terms of different signs and it is not clear how to manipulate them into a similar form , even though they have to be equivalent.
\end{enumerate}
\end{rem} 
Before we can prove \autoref{thm:probabilistic-heat-content-formula}, we state two preparatory lemmas collecting useful properties of the coefficients $\alpha_n(t)$:
\begin{lemma}
	\label{lem:alpha_n_1}
	The functions $\mathbb{N}_0 \times (0, \infty) \ni (n,t) \mapsto \alpha_n(t) \in (0, \infty)$ in~\eqref{eq:defintition_alpha}, are continuous and strictly monotonously \emph{decreasing} in $t$ and strictly monotonously \emph{increasing} $n$, and satisfy
	\[
	\lim_{t \rightarrow 0^+}
	\alpha_n(t)
	=
	\ell,
	\quad
	\lim_{t \rightarrow +\infty}	
	\alpha_n(t)
	=
	0,
	\quad
	\text{and}
	\quad
	\lim_{n \to \infty}
	\alpha_n 
	=
	\ell.
	\]
\end{lemma}

\begin{proof}
	This is an immediate consequence of almost sure positivity of $\eta_0$, of the definition of the $\eta_{k} - \eta_{k-1}$, and of~\eqref{eq:defintition_alpha}.
\end{proof}
\begin{lemma}\label{lem:re-arrangement}
Let $(\mu_n)_{n \in \mathbb{N}_0} \in \ell^1(\mathbb{N}_0)$ be a summable sequence of non-negative numbers and define $\beta_n := \sum_{j=n}^\infty \mu_j$ for $n \in \mathbb{N}_0$. Then:
\begin{align}\label{eq:re-arrangement}
\alpha_0(t)\beta_0 + \sum_{n=1}^\infty \big(\alpha_n(t)-\alpha_{n-1}(t) \big)\beta_n = \sum_{n=0}^\infty \mu_n \alpha_n(t) \quad \text{for all $t>0$.}
\end{align}
\end{lemma}
\begin{rem}
Note that by \autoref{lem:alpha_n_1}, the sequence $(\alpha_n(t))_{n \in \mathbb{N}_0}$ is strictly monotonously decrasing in $n$ (for fixed $t>0$) and thus this sequence is bounded by \autoref{thm:densities-random-walk}. Hence, the right-hand side of \eqref{eq:re-arrangement} indeed exists.
\end{rem}
\begin{proof}
Since $\beta_{n-1} = \mu_{n-1} + \beta_n$ for every $n \in \mathbb{N}$ and the $\alpha_n(t)$ are monotonously increasing in $n \in \mathbb{N}$, we obtain, since the left-hand side of \eqref{eq:re-arrangement} can be re-written as a telecopic series,
\begin{align*}
	\alpha_0(t)\beta_0 + \sum_{n=1}^N (\alpha_n(t)-\alpha_{n-1}(t))\beta_n 
	&= 
	\alpha_0(t)\beta_0 + \sum_{n=1}^N \big(\alpha_{n}(t)\beta_n - \alpha_{n-1}(t)\beta_{n-1} \big) + \sum_{n=1}^N \alpha_{n-1}(t) \mu_{n-1} 
	\\
	&= \alpha_N(t) \beta_N + \sum_{n=0}^{N-1} \alpha_n(t)\mu_n, 
\end{align*}
for every $N \in \mathbb{N}$. 
Sending $N \to \infty$ yields the statement, taking into account that $(\mu_n)_{n \in \mathbb{N}_0} \in \ell^1(\mathbb{N}_0)$ and $\alpha_N(t) \leq \ell$ imply $\alpha_N(t)\beta_N \rightarrow 0$.
%
\end{proof}
We now have everything at hand to prove \autoref{thm:probabilistic-heat-content-formula}.
\begin{proof}[Proof of \autoref{thm:probabilistic-heat-content-formula}]
	Starting with~\eqref{eq:Feynman-Kac}, we decompose the set of all Brownian paths hitting the Dirichlet vertex $\mv_\mathrm{D}$ (the set of paths never hitting $\mv_\mathrm{D}$ has probability zero) according to their \emph{combinatorial first hitting time} of $\mv_\mathrm{D}$, that is the smallest $n \in \mathbb{N}_0 = \{0,1,2, \dots \}$ such that $\mv_n = \mv_{\mathrm{D}}$, and then distinguish along the vertex $\mv_0$ to be hit first
\begin{align}
	\label{eq:heatcont-matthias}
\begin{aligned}
	\heatcont(\Graph; \mv_\mathrm{D} )
	&= 
	\int_\Graph
	\PP_x
	\left[
		\tau_{\mv_\mathrm{D}}^{\mathrm{BM}} \geq t 
	\right] 
	\ud 
	x 
	= 
	\sum_{\me \in \mE}
	\int_{0}^{\ell_\me}
	\E_x
	\left[	
	\sum_{n = 0}^\infty
	\mathbf{1}_{\{\eta_n = \tau_{\mv_\mathrm{D}}^{\mathrm{BM}} \}}
	\mathbf{1}_{\{ \eta_n \geq t \}}
	\right]
	\mathrm{d} x
	\\
	&=
	\sum_{n = 0}^\infty
    \sum_{\me \in \mE}
    \sum_{\mv \sim \me}
    \int_0^{\ell}
	\E_x
	\left[
		\mathbf{1}_{\{\mv_0 = \mv \}}
		\mathbf{1}_{\big\{ \mv_n = B_{\tau_{\mv_\mathrm{D}}^{\mathrm{BM}}}\big\} }
		\mathbf{1}_{\{\eta_n \geq t\}}
	\right]
	\ud x.
	\end{aligned}
	\end{align}
	The integration from $0$ to $\ell$ can be interpreted as an expectation with respect to another uniformly distributed random variable (multiplied with the factor $\ell$), determining the starting position on the interval.
	Writing $\PP_{\me}$ for the probability with respect to the joint distribution and conditioning on $\{ \mv_0 = \mv \}$, we have
	\begin{align*}
	\heatcont(\Graph; \mv_\mathrm{D} )
	&=
	\sum_{n = 0}^\infty
    \sum_{\me \in \mE}
    \sum_{\mv \sim \me}
    \int_0^{\ell}
    \PP_{\me}
    \left[
		\mathbf{1}_{\big\{\mv_n = B_{\tau_{\mv_\mathrm{D}}^{\mathrm{BM}}}\big\}}
		\mathbf{1}_{\{\eta_n \geq t\}}
		\: \Bigg\vert \:
		\mathbf{1}_{\{\mv_0 = \mv\}}
	\right]
	\PP_{\me}
	\left[
		\mv_0 = \mv
	\right] \ud x
	\\
	&=
	\sum_{n = 0}^\infty
	\sum_{\mv \in \mV}
	\frac{\deg(\mv)}{2}
	\, \ell \,
	\PP_{\me}
    \left[
		\mathbf{1}_{\big\{\mv_n = B_{\tau_{\mv_\mathrm{D}}^{\mathrm{BM}}}\big\}}
		\mathbf{1}_{\{\eta_n \geq t\}}
		\: \Bigg\vert \:
		\mathbf{1}_{\{\mv_0 = \mv\}}
	\right]
	\end{align*}
	where we used that every starting vertex $\mv = \mv_0$ appears exactly $\deg(\mv)$ many times in the sum and that $\PP_{\me} [ \mv_0 = \mv	] = \frac{1}{2}$ if $\mv$ is an end point of $\me$. 
	Now, the events 
	\[
	\left\{ \mv_n = B_{\tau_{\mv_\mathrm{D}}^{\mathrm{BM}}} \right\} \quad \text{and} \quad \{ \eta_n \geq t \}
	\]
	are stochastically independent in the $\sigma$-algebra, generated by $\{ \mv_0 = \mv \}$, hence, writing $\PP$ for the  probability measure with respect to the random variables $\eta_j$ (which do in particular not depend on the edge $\me$), and $\PP_{\mv_0}$ for the probability measure associated to the discrete random walk $(\mv_0, \mv_1, \dots)$ starting at $\mv_0$, introduced in \autoref{def:symmetric-random-walk}, we have
	\begin{align*}
	\heatcont(\Graph; \mv_\mathrm{D} )
	&=
	\sum_{n = 0}^\infty
	\sum_{\mv_0 \in \mV}
	\frac{\deg( \mv_0 )}{2} \PP_{\mv_0} 
	\left[
		\tau_{\mv_\mathrm{D}} = n
	\right]
	\ell 
	\,
	\PP[\eta_n \geq t]
	\\&=
	\sum_{n = 0}^\infty
	\sum_{\mv_0 \in \mV}
	\frac{\deg(\mv_0)}{2}
	\PP_{\mv_0} 
	\left[
		\tau_{\mv_\mathrm{D}} = n
	\right]
	\alpha_n(t),
	\end{align*}
	since $\eta_n = \eta_0 + \sum_{k=1}^n (\eta_k-\eta_{k-1})$ for every $n \in \mathbb{N}_0$ and thus (as $\eta_n$ does not depend on $\me$, that is, not on $x \in \me$)
	\begin{align*}
	\alpha_n(t) = \int_0^\ell \PP[\eta_n \geq t] \ud x = \ell \, \PP[\eta_n \geq t] \quad \text{for every $t>0$.}
	\end{align*}
	Next, we evaluate on the probability $\PP_{\mv_0}[\tau_{\mv_\mathrm{D}} \geq n]$ by summing over the individual probabilities of all paths $(\mv_0, \mv_1, \dots, \mv_n = \mv_\mathrm{D})$ emanating from $\mv_0$ and hitting the Dirichlet vertex $\mv_\mathrm{D}$ for the first time in the $n$-th step. 	Noting that a path $(\mv_0,\mv_1,\dots \mv_n=\mv_\mathrm{D})$ has probability
	\[
	\prod_{j = 0}^{n-1}
		\frac{1}{\deg(\mv_j )}
	\]
	we obtain
	\begin{align}\label{eq:combinatorial-prob-heat-content}
	\begin{aligned}	
	\heatcont(\Graph; \mv_\mathrm{D} )
	&=
	\sum_{n = 0}^\infty
	\sum_{\mv_0 \in \mV}
	\sum_{\substack{\text{$(\mv_0, \mv_1, \dots, \mv_n = \mv_\mathrm{D})$ path}\\
    \text{with $\mv_{\mathrm{D}} \not \in \{ \mv_1, \dots, \mv_{n-1}\}$}}}
    \frac{\deg( \mv_0 )}{2}
    \left(	
	\prod_{j = 0}^{n-1}
		\frac{1}{\deg(\mv_j )}
	\right)
	\alpha_n(t)
	\\
	&=
	\frac{1}{2}
	\left[
	\deg( \mv_{\mathrm{D}})
	\alpha_0(t)
	+
	\sum_{n = 1}^\infty
	\sum_{\mv_{\mathrm{D}} \neq \mv_0 \in \mV}
	\sum_{\substack{\text{$(\mv_0, \mv_1, \dots, \mv_n = \mv_\mathrm{D})$ path}\\
    \text{with $\mv_{\mathrm{D}} \not \in \{ \mv_1, \dots, \mv_{n-1}\}$}}}
    \left(	
	\prod_{j = 1}^{n-1}
		\frac{1}{\deg( \mv_j )}
	\right)
	\alpha_n(t)
	\right]
	\\
	&=
	\frac{\deg(\mv_{\mathrm{D}})}{2}
	\left[
	\alpha_0(t)
	+
	\sum_{n = 1}^\infty
	\sum_{
	\substack{
		\text{$(\mv_{\mathrm{D}} = \mw_0, \mw_1, \dots, \mw_n)$ path}\\
    \text{with $\mv_{\mathrm{D}} \not \in \{ \mw_1, \dots, \mw_{n}\}$}
    	}
	}
    \left(	
	\prod_{j = 0}^{n-1}
		\frac{1}{\deg(\mw_j)}
	\right)
	\alpha_n(t)
	\right]
	\end{aligned}
\end{align}
where in the last step, we flipped the orientation of each path and multiplied with $\deg(\mv_{\mathrm{D}} )$ in order to compensate the factor $\frac{1}{\deg( \mw_0 )}$ for $j = 0$ in the inverted paths.
However, the inner sum over the product of inverse degrees along paths is now equal to $\PP_{\mv_\mathrm{D}}[\tau_{\mv_\mathrm{D}} \geq n + 1]$, more precisely,
\begin{align*}
\sum_{
	\substack{
		\text{$(\mv_{\mathrm{D}} = \mw_0, \mw_1, \dots, \mw_n)$ path}\\
    \text{with $\mv_{\mathrm{D}} \not \in \{ \mw_1, \dots, \mw_{n}\}$}
    	}
	}
    \left(	
	\prod_{j = 0}^{n-1}
		\frac{1}{\deg(\mw_j)}
	\right) = \PP_{\mv_\mathrm{D}}[\tau_{\mv_\mathrm{D}} \geq n + 1] \quad \text{for any $n \in \mathbb{N}_0$,}
\end{align*}
whence,
\begin{align}\label{eq:pre-heat-content-formula}
	\heatcont(\Graph;  \mv_{\mathrm{D}} )
	&=
	\frac{\deg( \mv_{\mathrm{D}} )}{2}
	\left[
		\alpha_0(t)
		+
		\sum_{n = 1}^\infty
		\alpha_n(t)
		\PP_{\mv_\mathrm{D}}
		[
		\tau_{\mv_\mathrm{D}} \geq n + 1
		]
	\right].
\end{align}
Considering now the sequence $(\mu_n)_{n \in \mathbb{N}_0}$ defined by $\mu_n := \PP_{\mv_\mathrm{D}}[\tau_{\mv_\mathrm{D}} \geq n+1]$ for $n \in \mathbb{N}_0$, we observe
\begin{align*}
\sum_{n=0}^\infty \mu_n = \sum_{j=0}^\infty j \ \PP_{\mv_\mathrm{D}}[\tau_{\mv_\mathrm{D}} = j] = \mathbb{E}_{\mv_\mathrm{D}}[\tau_{\mv_\mathrm{D}}] = 2\# \mE,
\end{align*}
due to \autoref{prop:expectation_random_walk}, in particular $(\mu_n)_{n \in \mathbb{N}} \in \ell^1(\mathbb{N}_0)$. Therefore, \autoref{lem:re-arrangement} is applicable and eventually implies according to \eqref{eq:pre-heat-content-formula}
\begin{align*}
	\heatcont(\Graph;  \mv_{\mathrm{D}} ) 
	&=
	\frac{\deg( \mv_{\mathrm{D}} )}{2}
	\left[
	\alpha_0(t)
	\sum_{j = 0}^\infty
		\PP_{\mv_\mathrm{D}}
		[
		\tau_{\mv_\mathrm{D}} \geq j+1 
		]	
	+
	\sum_{n = 1}^\infty
	\left(
		\alpha_{n}(t) - \alpha_{n-1}(t)
	\right)
	\sum_{j = n}^\infty
		\PP_{\mv_\mathrm{D}}
		[
		\tau_{\mv_\mathrm{D}} \geq j+1 
		]
	\right]
	\\
	&=
	\frac{\deg( \mv_{\mathrm{D}} )}{2}
	\left[
	\alpha_0(t)
	\E_{\mv_{\mathrm{D}}}[\tau_{\mv_{\mathrm{D}}}]
	+
	\sum_{n = 1}^\infty
	\left(
		\alpha_{n}(t) - \alpha_{n-1}(t)
	\right)	
	\sum_{j = n+1}^\infty
		(j-n)
		\PP_{\mv_\mathrm{D}}
		[
		\tau_{\mv_\mathrm{D}} = j
		]
	\right]
	\\
	&=
	\frac{\deg( \mv_{\mathrm{D}} )}{2}
		\left[
		\alpha_0(t)
		\E_{\mv_{\mathrm{D}}}[\tau_{\mv_{\mathrm{D}}}]
		+
		\sum_{n = 1}^\infty
		\left(
		\alpha_{n}(t) - \alpha_{n-1}(t)
		\right)
		\E_{\mv_{\mathrm{D}}}\left[(\tau_{\mv_{\mathrm{D}}} -n)\mathbf{1}_{\{\tau_{\mv_{\mathrm{D}} \geq n+1}\} } \right]		
		\right].
\end{align*}
	This completes the proof.
\end{proof}
\begin{rem}
Note that, interpreting $\prod_{j=0}^{-1} \frac{1}{\deg(\mw_j)} = 1$ as the empty product, \eqref{eq:combinatorial-prob-heat-content} can be written in the shortened form
\begin{align}
\heatcont(\Graph;\{ \mv_\mathrm{D} \}) = \frac{\deg(\mv_\mathrm{D})}{2} \sum_{n=0}^\infty \sum_{\mv_0 \in \mV}\sum_{
	\substack{
		\text{$(\mv_{\mathrm{D}} = \mw_0, \mw_1, \dots, \mw_n)$ path}\\
    \text{with $\mv_{\mathrm{D}} \not \in \{ \mw_1, \dots, \mw_{n}\}$}
    	}
	}
    \left(	
	\prod_{j = 0}^{n-1}
		\frac{1}{\deg(\mw_j)}
	\right)
	\alpha_n(t). 
\end{align}
We have decomposed the heat content of a compact finite metric graph $\Graph$ into an infinite sum consisting of terms completely determined by the underlying combinatorial graph $\mG$ and completely one-dimensional terms given by the $\alpha_n(t)$.
\end{rem}

Before turning to the proof of \autoref{thm:main-thm}~\eqref{item:fk-small-thm}, let us discuss once again the quantities $\alpha_n(t)$ appearing in the previous theorem.
First note that the $\alpha_n(t)$ do not depend on the topology of the graph.

Using the aforementioned trick to consider the integral over $x \in [0,\ell]$ in~\eqref{eq:defintition_alpha} as another random variable (up to the normlalization factor $\ell$), independent of the $\eta_j$ and distributed uniformly in $[0, \ell]$.
Defining $X_0$ to be the average value of $\eta_0$, as well as $X_n := \eta_n - \eta_{n-1}$ for $n \in \mathbb{N}$, we can write
\begin{equation}
	\label{eq:terms_in_alpha_independent}
	\alpha_n(t)
	=
	\ell
	\cdot
	\PP
	\left[
		X_0 + X_1 + \dots + X_n \geq t
	\right].	
\end{equation}

The next lemma states that the differences between the $\alpha_k(t)$ are hierarchical in the sense that $\alpha_{k+1}(t) - \alpha_k(t)$, indeed even $\ell - \alpha_k(t)$, will be small compared to $\alpha_k(t) - \alpha_{k-1}(t)$ as $t \to 0$.
This implies that terms associated with shorter paths in the path expansion of Theorem~\ref{thm:probabilistic-heat-content-formula} will strictly dominate the expressions associated with longer paths in the small-time regime.

\begin{lemma}
	\label{lem:alpha_n_2}
	We have for all $k \in \mathbb{N}$
	\[
	\lim_{t \to 0^+}	
	\frac{\ell - \alpha_{k}(t)}{\alpha_{k}(t) - \alpha_{k - 1}(t)}
	=
	0.
	\]
\end{lemma}
	
\begin{proof}
	We use~\eqref{eq:terms_in_alpha_independent} to write
	\begin{align*}
	\frac{\ell - \alpha_{k}(t)}{\alpha_{k}(t) - \alpha_{k-1}(t)}
	&=
	\frac
	{\PP \left[ X_0 + \dots + X_{k} < t \right]}
	{\PP  \left[ X_0 + \dots + X_{k-1} < t \right] -  \PP\left[ X_0 + \dots + X_{k} < t \right]}
	\\&=
	\frac
	{\PP [ X + Y < t]}
	{\PP [X < t] - \PP[X + Y < t] }
	\end{align*}
	for independent, nonnegative random variables $X := X_0 + \dots + X_{k-1}$ and $Y := X_k$.
	But since
	\begin{align*}
		\PP[X+Y \leq t] 
		= 
		\int_0^t \int_0^s \ud \rho_x \ud \rho_y \leq \int_0^t \mathbb{P}[X \leq s] \ud \rho_y \leq 					\mathbb{P}[X \leq t] \mathbb{P}[Y \leq t], 
\end{align*}
	we can further estimate
	\[
	\frac
	{\PP [ X + Y < t]}
	{\PP [X < t] - \PP[X + Y < t] }
	=
	\left(
		\frac{\PP [X < t]}{\PP [X + Y < t]} - 1
	\right)^{-1}
	\leq
	\left(
	\frac{1}{\PP [Y < t]} - 1
	\right)^{-1}
	\to 0,
	\]
	as $t \rightarrow 0^+$, since $Y = X_k = \eta_k - \eta_{k-1}$ and
    \[
    \lim_{t \to 0^+} \PP [Y < t] = \lim_{t \to 0^+} \PP [\eta_k - \eta_{k-1} < t] = 0, 
    \]	
    by \eqref{eq:limit-eta-n-in-difference}.
\end{proof}

We are now ready to prove~\autoref{thm:main-thm}~\eqref{item:fk-small-thm}.

\begin{proof}[Proof of~\autoref{thm:main-thm}~\eqref{item:fk-small-thm}]
Since $\Graph$ has rationally dependent edge lengths, we may assume -- after inserting artificial vertices of degree two -- that $\Graph$ is an equilateral metric graph constructed over a combinatorial graph $\mG$ with $N$ edges of length $\ell$, in particular, $N\ell = \vert \Graph \vert$. 
Moreover, since the heat content is decreasing with respect to the degree of $\mv_\mathrm{D}$ (cf.\ \cite{BifMug24}), we suppose without loss of generality 
\[
\deg_\mG(\mv_\mathrm{D}) = 1. 
\]
Let $\mathcal{P}_{\lvert \Graph \rvert}$ be the corresponding path graph of same total length and also constructed over a combinatorial path graph $\mP$ with $N$ edges having length $\ell$. 
As indicated above, we use the notation $\tau_{\mv_{\mathrm{D}}}^\mG$ for the first hitting time of the symmetric random walk in $\mG$, and use $\tau_{\mv_{\mathrm{D}}}^\mP$ to denote the corresponding hitting time of the $\{0\}$-vertex in the combinatorial path graph $\mP$ corresponding to $\mathcal{P}_{\lvert \Graph \rvert}$.
Now, if $\Graph$ is not a path graph, there must be a vertex with a degree of at least $3$.
Let $k_0$ be twice the minimal combinatorial distance of $\mv_{\mathrm{D}}$ to such a vertex of degree at least $3$, more precisely,
\[
\frac{k_0}{2} = \min_{\mw \in \mV(\mG), \: \deg_\mG(\mw) \geq 3} \mathrm{dist}_\mG(\mv_\mathrm{D},\mw).
\]
Since $\deg(\mv_{\mathrm{D}} ) = 1$, the unique path between $\mv_{\mathrm{D}}$ and this minimal vertex is isomorphic to a path. The leads us to (cf.\ also \autoref{fig:discrete-random-walk} below): 
\begin{align*}
\mathbb{P}_{\mv_\mathrm{D}}[\tau_{\mv_{\mathrm{D}}}^\mG = k] &= \mathbb{P}_{\mv_\mathrm{D}}[\tau_{\mv_{\mathrm{D}}}^\mP = k] \:\: \text{for all $k \leq k_0 - 1$; 	and} \:\:  \mathbb{P}_{\mv_\mathrm{D}}[\tau_{\mv_{\mathrm{D}}}^\mG = k_0] < \mathbb{P}_{\mv_\mathrm{D}}[\tau_{\mv_{\mathrm{D}}}^\mP = k_0].
\end{align*}

\begin{figure}[ht]
	\begin{tikzpicture}
	\begin{scope}
		\draw[thick] (0,0) -- (3,0);
		\draw[thick] (3,0) -- (3.7,.7);		
		\draw[thick] (3,0) -- (3.7,-.7);		
		
		\draw[thick, fill = white] (0,0) circle	(3pt);	
		\draw[thick, fill = black] (1,0) circle	(3pt);
		\draw[thick, fill = black] (2,0) circle	(3pt);
		\draw[thick, fill = black] (3,0) circle	(3pt);
		\draw[thick, fill = black] (3.7,.7) circle	(3pt);		
		\draw[thick, fill = black] (3.7,-.7) circle (3pt);	
		
		\draw (0,-.45) node {$\mv_\mathrm{D}$};
		\draw (3,-.45) node {$\frac{k_0}{2}$};		
		
		\draw[->] (3.1,.25) arc (180:115:.6);
		\draw (3.25,1.1) node {$\frac{1}{3}$};
		\draw[->] (3.25,0) arc (90:15:.6);
		\draw (4.1,-.25) node {$\frac{1}{3}$};		
		\draw[->] (2.9,.25) arc (45:135:.6);		
		\draw (2.1,.8) node {$\frac{1}{3}$};			
	\end{scope}
	
	\begin{scope}[xshift = 6cm]
		\draw[thick] (0,0) -- (5,0);
		
		\draw[thick, fill = white] (0,0) circle	(3pt);	
		\draw[thick, fill = black] (1,0) circle	(3pt);
		\draw[thick, fill = black] (2,0) circle	(3pt);
		\draw[thick, fill = black] (3,0) circle	(3pt);
		\draw[thick, fill = black] (4,0) circle	(3pt);		
		\draw[thick, fill = black] (5,0) circle (3pt);	
		
		\draw (0,-.45) node {$\mv_\mathrm{D}$};
		\draw (3,-.45) node {$\frac{k_0}{2}$};	
	
		\draw (2.5,.8) node {$\frac{1}{2}$};	
		\draw[->] (2.9,.25) arc (45:135:.6);				
		\draw (3.5,.8) node {$\frac{1}{2}$};			
		\draw[->] (3.1,.25) arc (135:45:.6);		
				
	\end{scope}

	\end{tikzpicture}
\caption{Illustration of the choice of $k_0$ in the proof of Theorem~\ref{thm:main-thm} in the case $k_0 = 6$.
We compare the path graph on the right and a non-path graph on the left and consider the unique realization of a random walk of minimal length, originating from $\mv_D$, going $k_0/2$ steps to the left and then returning to $\mv_D$ in $k_0/2$ steps.
On the graphs, all individual jumps have the same probability except for the $k_0/2$-th step. In it, the probability of jumping back is larger (namely $1/2$) in the path graph than it is in the other graph (in this case $1/3$).
}\label{fig:discrete-random-walk}	
\end{figure}
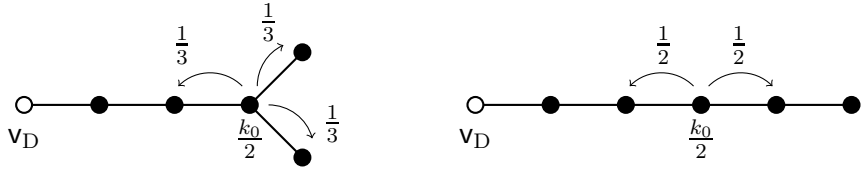

Now put
\[
C({\mP, \mG}) := \mathbb{P}_{\mv_\mathrm{D}}[\tau_{\mv_{\mathrm{D}}}^\mP = k_0] - \mathbb{P}_{\mv_\mathrm{D}}[\tau_{\mv_{\mathrm{D}}}^\mG = k_0] > 0.
\]
Subtracting \eqref{eq:probabilistic_expression_heat_content} for $\heatcont(\Graph; \mv_{\mathrm{D}} )$ from $\heatcont(\mathcal{P}; \{ 0 \})$, we obtain according to \autoref{thm:probabilistic-heat-content-formula} (note that the $\alpha_n(t)$ do \emph{not} depend on the topology of the underlying graph)
\begin{equation}\label{eq:comparison-formula}
	\begin{aligned}
		&\heatcont(\mathcal{P}_{\vert \Graph \vert}; \{ 0 \})
		-		
		\heatcont(\Graph; \mv_{\mathrm{D}} ) 	
		\\
		\qquad\quad
		&			
		= 
		\frac{1}{2}\sum_{n = 1}^\infty
		\left(
		\alpha_{n}(t) - \alpha_{n-1}(t)
		\right)
		\Big( 
			\mathbb{E}_{\mv_{\mathsf{D}}}\big[(\tau_{\mv_{\mathsf{D}}}^\mP - n) \mathbf{1}_{\{\tau_{\mv_{\mathsf{D}}}^\mP \geq n+1\}} \big] 
			-
			\mathbb{E}_{\mv_{\mathsf{D}}}\big[(\tau_{\mv_{\mathsf{D}}}^\mG - n) \mathbf{1}_{\{\tau_{\mv_{\mathsf{D}}}^\mG \geq n+1\}}  \big] 
		\Big)  
	\end{aligned}
\end{equation}
Using $\mathbb{E}_{\mv_\mathrm{D}}\big[\tau_{\mv_{\mathrm{D}}}^\mP \big] = \mathbb{E}_{\mv_\mathrm{D}}\big[\tau_{\mv_{\mathrm{D}}}^\mG \big] = 2 \# \mE$ due to \autoref{prop:expectation_random_walk}, and using $\mathbf{1}_{ \{\tau_{\mv_{\mathrm{D}}}^\mP \geq n + 1 \}} = 1 - \mathbf{1}_{ \{\tau_{\mv_{\mathrm{D}}}^\mP \leq n \}}$, we find for all $n \in \N$
	\begin{equation}
		\label{eq:first_differences_terms}
		\begin{aligned}
		&\mathbb{E}_{\mv_\mathrm{D}}\big[(\tau_{\mv_{\mathrm{D}}}^\mP - n) \mathbf{1}_{ \{\tau_{\mv_{\mathrm{D}}}^\mP \geq n + 1 \}} \big] - 
			\mathbb{E}_{\mv_\mathrm{D}}\big[(\tau_{\mv_{\mathrm{D}}}^\mG - n) \mathbf{1}_{\{ \tau_{\mv_{\mathrm{D}}}^\mG \geq n + 1\}} \big] 
		\\
		&= 
		\sum_{k = 0}^{n}
			(n - k)
			\left(
			\mathbb{P}_{\mv_{\mathrm{D}}}
				\big[\tau_{\mv_{\mathrm{D}}}^\mP = k \big]
				-
			\mathbb{P}_{\mv_{\mathrm{D}}}
				\big[\tau_{\mv_{\mathrm{D}}}^\mG = k \big] 
			\right).
		\end{aligned}
	\end{equation}
Due to the choice of $k_0$, this is zero for all $n \leq k_0$, whence, according to~\eqref{eq:first_differences_terms}
	\begin{align*}
	&\heatcont(\mathcal{P}_{\vert \Graph \vert}; \{ 0 \})
		-		
	\heatcont(\Graph; \mv_{\mathrm{D}} ) 	
	\\
	&=
	\frac{1}{2}
		\sum_{n = k_0 + 1}^\infty
		\left(
		\alpha_{n}(t) - \alpha_{n-1}(t)
		\right)
		\Big( 
			\mathbb{E}_{\mv_{\mathsf{D}}}\big[(\tau_{\mv_{\mathsf{D}}}^\mP - n) \mathbf{1}_{\{\tau_{\mv_{\mathsf{D}}}^\mP \geq n+1\}} \big] 
			-
			\mathbb{E}_{\mv_{\mathsf{D}}}\big[(\tau_{\mv_{\mathsf{D}}}^\mG - n) \mathbf{1}_{\{\tau_{\mv_{\mathsf{D}}}^\mG \geq n+1\}}  \big] 
		\Big) 
		\\
		&\geq
		\frac{1}{2}
		\left(
		\alpha_{k_0 + 1}(t) - \alpha_{k_0}(t)
		\right)
		\Big( 
			\mathbb{E}_{\mv_{\mathsf{D}}}\big[(\tau_{\mv_{\mathsf{D}}}^\mP - n) \mathbf{1}_{\{\tau_{\mv_{\mathsf{D}}}^\mP \geq k_0+2\}} \big] 
			-
			\mathbb{E}_{\mv_{\mathsf{D}}}\big[(\tau_{\mv_{\mathsf{D}}}^\mG - n) \mathbf{1}_{\{\tau_{\mv_{\mathsf{D}}}^\mG \geq k_0+2\}}  \big] 
		\Big) 
		\\
		&\qquad
		-
		N \sum_{n=k_0 + 2}^\infty \big( \alpha_n(t) - \alpha_{n-1}(t) \big) 
		.
	\end{align*}	
	Furthermore, again by~\eqref{eq:first_differences_terms}
\begin{align*}
				&\mathbb{E}_{\mv_\mathrm{D}}\big[(\tau_{\mv_{\mathrm{D}}}^\mP - (k_0 + 1)) \mathbf{1}_{ \{\tau_{\mv_{\mathrm{D}}}^\mP \geq k_0 + 2 \}} \big] - 
			\mathbb{E}_{\mv_\mathrm{D}}\big[(\tau_{\mv_{\mathrm{D}}}^\mG - (k_0 + 1)) \mathbf{1}_{\{ \tau_{\mv_{\mathrm{D}}}^\mG \geq k_0 + 2\}} \big] 
		\\
		&= 
		\sum_{k = 0}^{k_0 + 1}
			(k_0 + 1 - k)
			\left(
			\mathbb{P}_{\mv_{\mathrm{D}}}
				\big[\tau_{\mv_{\mathrm{D}}}^\mP = k \big]
				-
			\mathbb{P}_{\mv_{\mathrm{D}}}
				\big[\tau_{\mv_{\mathrm{D}}}^\mG = k \big] 
			\right)
		\\
		&\geq
		\mathbb{P}_{\mv_{\mathrm{D}}}
				\big[\tau_{\mv_{\mathrm{D}}}^\mP = k_0 \big]
				-
			\mathbb{P}_{\mv_{\mathrm{D}}}
				\big[\tau_{\mv_{\mathrm{D}}}^\mG = k_0 \big] 
		=
		C(\mP, \mG) > 0,
\end{align*}
and we infer
\begin{equation}
	\label{eq:lower_bound_difference_simplified}
	\heatcont(\mathcal{P}_{\vert \Graph \vert}; \{ 0 \})
		-		
	\heatcont(\Graph; \mv_{\mathrm{D}} ) 	
	\geq
	C(\mP, \mG)	
	(\alpha_{k_0 + 1}(t) - \alpha_{k_0}(t))
	+
	N
	(\ell - \alpha_{k_0 + 1}(t)).
\end{equation}
By \autoref{lem:alpha_n_2}, the right-hand side of \eqref{eq:lower_bound_difference_simplified} becomes strictly positive for sufficiently small $t$.

Conversely, if $\heatcont(\Graph;\mv_\mathrm{D}) = \heatcont(\mathcal{P}_{\vert \Graph \vert}; \{0\})$ for all sufficiently small $t>0$, the above argumentation prohibits the existence of a vertex $\mw \in \mV$ with $\deg_\mG(\mw) \geq 3$. 
Thus, $\deg_{\mG}(\mw) \in \{1,2\}$ for all $\mw \in \mV$ which implies that $\mG = \mP$ and thus, $\Graph = \mathcal{P}_{\vert \Graph \vert}$. This finishes the proof.
\end{proof}

\begin{rem}
Note that $t_0 > 0$ in \autoref{thm:main-thm} is implicitely given by the largest $t_0$ such that
\[
\frac{C(\mP, \mG)}{2N} > \frac{\ell-\alpha_{k_0 + 1}(t_0)}{\alpha_{k_0 + 1}(t_0)-\alpha_{k_0}(t_0)}.
\] 
\end{rem}

Note that the calculations made in \eqref{eq:comparison-formula} and \eqref{eq:first_differences_terms} yield the following comparison principle of two different heat contents.
\begin{corollary}[Comparison formula for the heat content]
Let $\Graph$ and $\mathcal{H}$ be two equilateral metric graphs with the same number of edges and with Dirichlet vertices $\mv_{\mathrm{D},\Graph}$ and $\mv_{\mathrm{D},\mathcal{H}}$ of the same degree $m_\mathrm{D} \in \mathbb{N}$. 
Then, for all $t > 0$,
\begin{align}\label{eq:deviation-formula}
\begin{aligned}
&\heatcont(\Graph;\mv_{\mathrm{D},\Graph}) - \heatcont(\mathcal{H};\mv_{\mathrm{D},\mathcal{H}}) \\& \qquad=\frac{m_\mathrm{D}}{2}
		\sum_{n = 1}^\infty
		\left(
		\alpha_{n}(t) - \alpha_{n-1}(t)
		\right) \sum_{k=0}^n (n-k)\Big( \mathbb{P}_{\mv_{\mathrm{D}, \Graph}}\big[\tau_{\mv_{\mathrm{D}, \Graph}}^\mG = k \big] - \mathbb{P}_{\mv_{\mathrm{D}, \mathcal{H}}}\big[\tau_{\mv_{\mathrm{D}, \mathcal{H}}}^\mathsf{H} = k \big] \Big).		
\end{aligned}
\end{align}
\end{corollary}

\section{Mercer's theorem for the heat content and Faber--Krahn inequality for large times}
\label{sec:proof_large_times}

In this section we prove~\autoref{thm:main-thm}~\eqref{item:fk-large-thm}, that is the Faber--Krahn inequality for the heat content at large times. In contrast to the probabilistic methods of the previous section, we rely on spectral theoretic approach using Mercer's theorem. Denoting the eigenvalues (counted with multiplicities) by
\begin{align}\label{eq:eigenvalues}
0 < \lambda_1(\Graph;\mv_\mathrm{D}) < \lambda_2(\Graph;\mv_\mathrm{D}) \leq \lambda_3(\Graph;\mv_\mathrm{D}) \leq \dots \rightarrow +\infty,
\end{align}
and a corresponding orthonormal basis of eigenfunctions of $-\Delta_{\Graph}^{\mv_{\mathrm{D}}}$ by and $(\phi_k)_{k \in \N} \subset L^2(\Graph)$, respectively, spectral calculus immediately implies
\begin{equation}
	\label{eq:Mercer}
	\heatcont(\Graph; \mv_\mathrm{D})
	=
	\sum_{k=1}^\infty
	\mathrm{e}^{- t \lambda_k(\Graph;\mv_\mathrm{D})}
	\big\vert \left\langle \phi_k, \mathbf{1} \right\rangle_{L^2(\Graph)} \big\vert^2 \quad \text{for all $t>0$.}
\end{equation}
This expression is due to Mercer, see also~\cite{PaulsenR-16} and references therein for a more comprehensive overview.
It can be used to immediately infer property \eqref{item:fk-large-thm} of~\autoref{thm:main-thm}:
\begin{proof}[Proof of Theorem~\ref{thm:main-thm}.\eqref{item:fk-large-thm}]
	It is known~\cite{Nic87, Fri05} that among all connected graphs of total length $\lvert \Graph \rvert$ with at least one Dirichlet vertex the lowest eigenvalue $\lambda_k(\Graph;\mv_\mathrm{D})$ is minimized by the lowest eigenvalue 
	\begin{align*}
	\lambda_1(\mathcal{P}_{\lvert \Graph \rvert};\{ 0 \})
	=
	\frac{\pi^2}{4\lvert \Graph \rvert^2}
	\end{align*}
	of the path graph $\mathcal{P}_{\lvert \Graph \rvert}$, that is the interval with a Dirichlet condition on one side and a Neumann condition on the other side with equality if and only if $\Graph = \mathcal{P}_{\lvert \Graph \rvert}$.
	Therefore, if $\Graph \neq \mathcal{P}_{\lvert \Graph \rvert}$, it follows that 
	\begin{align}\label{eq:lowest-eigenvalue-optimal}
	\lambda_1(\Graph;\mv_\mathrm{D})
	>
	\lambda_1(\mathcal{P}_{\lvert \Graph \rvert};\{0\}).
	\end{align}
	Denoting the first normalized eigenfunction of $\Delta_{\mathcal{P}_{\lvert \Graph \rvert}}^{\{0\}}$ by $\psi$, we can estimate 
	\begin{align}\label{eq:mercer-estimate}
	\begin{aligned}
	&\heatcont(\mathcal{P}_{\lvert \Graph \rvert}; \{0\})
	-
	\heatcont(\Graph; \mv_{\mathrm{D}})
	\\& \qquad\quad\geq
	\mathrm{e}^{- t \lambda_1(\mathcal{P}_{\lvert \Graph \rvert};\{0\})}
	\left(
		\underbrace{\lvert \left\langle \psi, \mathbf{1} \right\rangle \rvert^2}_{> 0}
		-
		\sum_{k = 1}^\infty 
		\mathrm{e}^{- t \left(\lambda_k(\Graph;\mv_\mathrm{D}) - \lambda_1(\mathcal{P}_{\lvert \Graph \rvert};\{0\}) \right)}
		\lvert \left\langle \phi_k, \mathbf{1} \right\rangle \rvert^2
	\right)
	\\
	& \qquad\quad\geq
	\mathrm{e}^{- t \lambda_1(\mathcal{P}_{\lvert \Graph \rvert};\{0\})}
	\left(
		\lvert \left\langle \psi, \mathbf{1} \right\rangle \rvert^2
		-
		\mathrm{e}^{- t \left(\lambda_1(\Graph;\mv_\mathrm{D}) - \lambda_1(\mathcal{P}_{\lvert \Graph \rvert};\{0\}) \right)}
		\underbrace{\sum_{k = 1}^\infty
		\lvert \left\langle \phi_k, \mathbf{1} \right\rangle \rvert^2}_{= \lvert \Graph \rvert}
	\right).
	\end{aligned}
	\end{align}
	But since 
	\[
	\mathrm{e}^{- t \left(\lambda_1(\Graph;\mv_\mathrm{D}) - \lambda_1(\mathcal{P}_{\lvert \Graph \rvert};\{0\}) \right)} < \frac{\lvert \left\langle \psi, \mathbf{1} \right\rangle \rvert^2}{\lvert \Graph \rvert}
	\]
	for sufficiently large $t$, the difference will be positive for such $t$.
	
	Moreover, if $\heatcont(\Graph;\mv_\mathrm{D}) = \heatcont(\mathcal{P}_{\vert \Graph \vert};\{0\})$ for all $t \geq T_0$, then $\Graph \neq \mathcal{P}_{\vert \Graph \vert}$ would imply according to \eqref{eq:lowest-eigenvalue-optimal} and \eqref{eq:mercer-estimate} that $\heatcont(\Graph;\mv_\mathrm{D}) < \heatcont(\mathcal{P}_{\vert \Graph \vert};\{0\})$ for $t>T_0$ large enough, thus leading to equality if and only if $\Graph = \mathcal{P}_{\vert \Graph \vert}$. This completes the proof.
	\end{proof}

\begin{rem}
	Even though expression \eqref{eq:Mercer} might look innocent the terms $\left\langle \phi_k, \mathbf{1} \right\rangle$ are hard to control and bound from below.
	Yet they become relevant, in particular at small times, justifying the use of different methods in the small and large time regime.
	 
\end{rem}

\end{document}